\documentclass{amsart}
\usepackage{graphics}

\newcommand{\A}{\mathcal A}
\newcommand{\M}{\mathbb M}
\newcommand{\R}{\mathbb R}
\newcommand{\F}{\mathbb F}
\newcommand{\Z}{\mathbb Z}

\newcommand{\Ext}{\mathrm{Ext}}
\newcommand{\C}{\mathbb C}
\newcommand{\toda}[1]{\langle #1\rangle}
\newcommand{\Spec}{\mathrm{Spec}}

\newtheorem{thm}{Theorem}[section]
\newtheorem{prop}[thm]{Proposition}
\newtheorem{lem}[thm]{Lemma}
\newtheorem{cor}[thm]{Corollary}
\newtheorem{conj}[thm]{Conjecture}
\newtheorem{defn}{Definition}
\newtheorem*{Remark}{Remark}

\theoremstyle{definition}

\title{Ext and the Motivic Steenrod Algebra over $\R$}
\author{Michael A.~Hill}

\begin{document}
\bibliographystyle{amsplain}

\begin{abstract}
We present a descent style, Bockstein spectral sequence computing $\Ext$ over the motivic Steenrod algebra over $\R$ and related sub-Hopf algebras. We demonstrate the workings of this spectral sequence in several examples, providing motivic analogues to the classical computations related to $BP\langle n\rangle$ and $ko$.
\end{abstract}

\maketitle

\section{Introduction}
Motivic stable homotopy theory intertwines classical algebraic geometry and stable algebraic topology. Rather than considering just the stable category of spaces, one considers a category of spaces built out of smooth schemes over a chosen ground field $k$. This framework, the foundations of which can be found in work of Morel and Voevodsky \cite{MoVo99}, allows one to make use of many of the classical tools and techniques from algebraic topology to solve problems like the Milnor conjecture. A nice introduction for topologists is found in work of Dugger and Isaksen \cite{DuIsCell05}.

One of the key features in motivic stable homotopy is that the spheres are bigraded. There is the simplicial sphere $S^{1,0}$ (which is related to the suspension tied to the triangulated structure), and there is a geometric sphere $S^{1,1}=\mathbb G_m$. Taking smash products of these spheres and formally desuspending them produces spheres $S^{p,q}$\!, $p,q\in\Z$, and this yields bigraded homotopy groups:
\[
\pi_{p,q}(X)=[S^{p,q},X].
\]

Much of the information known about these groups is due to Morel. He showed that $\pi_{p,q}(S^{0,0})=0$ if $p<q$ \cite{Mo05}. Moreover, he showed that $\pi_{n,n}(S^{0,0})$ is the Milnor-Witt $K$-theory of the ground field \cite{Mo04}.

Morel also established the foundations of the motivic Adams spectral sequence \cite{Mo99}. This is a trigraded spectral sequence of the form
\[
E_2=\Ext_{\A}^{s,t,u}\big(H^{\ast,\ast}(X),H^{\ast,\ast}(\Spec(k))\big)\Rightarrow\pi_{t-s,u}(X_{2}^{\widehat{}}),
\]
where $\A$ is the mod-$2$ motivic Steenrod algebra, where $H^{\ast,\ast}(-)$ is mod-$2$ motivic cohomology, and where $X_{2}^{\widehat{}}$ is the nilpotent completion of $X$ with respect to the motivic Eilenberg-Maclane spectrum $H\F_2$. The motivic Steenrod algebra and the motivic cohomology of a point were worked out by Voevodsky, and we will review the salient points below. Working over the ground field $\C$, Dugger and Isaksen have successfully implemented a research program to compute the relevant $\Ext$ groups and run the spectral sequence in low dimensions \cite{DuIs}.

This paper serves as a preamble to joint work with Dugger and Isaksen in which we will compute motivic stable homotopy groups using the Adams spectral sequence over $\R$ through a range. The primary tool we will use is a descent style spectral sequence which converts motivic $\Ext$ over $\C$ to motivic $\Ext$ over $\R$. In this paper, we compute several examples which demonstrates the subtleties of the full form. These examples has the added advantage of being related to motivic $BP$ and conjecturally related to the algebraic and Hermitian $K$-theories of $\R$. In particular, this computation should detect the motivic image of $J$. We will return to these applications in the last section.

\section{Background}
\subsection{Motivic Cohomology and the Motivic Steenrod Algebra}
For completeness and for reference, we summarize the basic results needed about $\M_2$, the motivic cohomology of a point, and of $\A^\ast$, the mod $2$ motivic Steenrod algebra. In this section and in all that follows, we will almost exclusively work over $\Spec(\R)$. In a few situations, we will also work over $\Spec(\C)$, and in those cases, we will use a superscript $\C$ to distinguish from the real case.
\begin{thm}[Voevodsky \cite{Vo03}] As an algebra,
\[
\M_2=\F_2[\tau,\rho],
\]
where $|\tau|=(0,1)$ and $|\rho|=(1,1)$.
\end{thm}
We will, in our description of the dual Steenrod algebra and homology, use the standard convention that $H^{\ast,\ast}=H_{-\ast,-\ast}$. We will continue to use the names $\tau$ and $\rho$.

Voevodsky also computed the motivic Steenrod algebra over $\R$ \cite{VoPower03, Vo07}. It is the associative algebra over $\M_2$ generated by classes $Sq^{2i}$ and $Sq^{2i-1}$\!, in bidegrees $(2i,i)$ and $(2i-1,i-1)$ respectively, subject to a substantially more complicated form of the Adem relations (which will not be needed directly). For us, the most important part is that $Sq^1\tau=\rho$. Since the ground ring $\M_2$ is not central, the dual Steenrod algebra $\A$ is a Hopf algebroid over $\M_2$.

\begin{thm}[Voevodsky \cite{VoPower03, Vo07}]
\[
\A=\big(\M_2,\M_2[\xi_1,\dots][\tau_0,\dots]/(\tau_i^2-\rho\tau_{i+1}-\rho\tau_0\xi_{i+1}-\tau\xi_{i+1})\big).
\]
The left unit is the canonical inclusion, while the right unit is given by $\eta_R(\rho)=\rho$ and $\eta_R(\tau)=\tau+\tau_0\rho$. The coproducts on the generators $\xi_i$ and $\tau_i$ are the classical coproducts. The bidegrees of the elements are given by
$|\xi_i| = (2^{i+1}-2,2^i-1)$ and $|\tau_i| = (2^{i+1}-1,2^i-1)$.
\end{thm}

Just as in the classical case, we also consider finitely generated subalgebras of the Steenrod algebra. In the dual case, these give Hopf algebroid quotients of $\A$ which are finitely generated over $\M_2$.

\begin{defn}\mbox{}
\begin{enumerate}
\item Let $\A(n)$ denote the quotient Hopf algebroid obtained by reducing modulo the ideal $(\xi_1^{2^n}, \xi_2^{2^{n-1}},\dots,\xi_n^2,\xi_{n+1},\dots)+(\tau_{n+1},\dots)$:
\[
\A(n)=\big(\M_2,\M_2[\xi_1,\dots,\xi_{n}][\tau_0,\dots,\tau_n]/(\xi_i^{2^{n-i+1}}, \tau_i^2-\rho\tau_{i+1}-\rho\tau_0\xi_{i+1}-\tau\xi_{i+1})\big).
\]

\item Let $E(n)$ denote the reduction modulo $(\xi_1,\dots)+(\tau_{n+1},\dots)$.
\[
E(n)=\big(\M_2,\M_2[\tau_0,\dots,\tau_n]/(\tau_i^2-\rho\tau_{i+1}, \tau_n^2)\big).
\]
\end{enumerate}
\end{defn}
The Hopf algebroid $\A(n)$ is dual to the sub-Hopf algebra of the motivic Steenrod algebra generated by $Sq^{2^i}$ for $i\leq n$, while $E(n)$ is dual to the sub-Hopf algebra generated by the Milnor primitives $Q_i$ for $i\leq n$.

\subsection{Bockstein Spectral Sequence}

The Hopf algebroid $\A$ has a distinguished invariant ideal: $(\rho)$. If we reduce modulo $\rho$, then we recover the complex motivic dual Steenrod algebra $\A^\C$. We therefore have a Miller-Novikov style Bockstein spectral sequence for undoing this reduction \cite{RaGB}.

\begin{prop}
There is a quadruply graded spectral sequence of algebras of the form
\[
E_1=\Ext_{\A^{\C}}(\M_2^\C,\M_2^\C)[\rho]\Longrightarrow\Ext_{\A}(\M_2,\M_2).
\]
\end{prop}

Since the motivic Steenrod algebra is bigraded, the $\Ext$ groups are trigraded. We will use the following Adams convention: $|x|=(t-s,s,u)$, where $t$ is the internal topological degree of any class representing $x$ in the cobar complex, $s$ is the cohomological degree, and $u$ is the motivic weight. The only algebra generator of non-zero Bockstein degree is $\rho$, so when we will only describe the aforementioned tridegrees, even in the Bockstein spectral sequence.

There are some immediate permanent cycles.

\begin{prop}\label{prop:PermCycles}
If $x\in \Ext_{\A^\C}(\M_2,\M_2)$ can be represented in the cobar complex using only polynomials in the classes $\tau_i$ and $\xi_i$, then $x$ survives the Bockstein spectral sequence.
\end{prop}

This is because the $\rho$-Bockstein spectral sequence measures the deviation of a cobar differential from being zero. If our classes are classical, then their bar differential is zero regardless of the presence or absence of $\rho$. We will use this to see that certain classes are permanent cycles.

Before continuing, we present a general lemma on the creation of new permanent cycles in Bockstein spectral sequences. This will help us streamline many of the computations.

\begin{lem}\label{lem:BocksteinCycles}
In a Bockstein spectral sequence associated to the filtration of an associative, homotopy-commutative differential graded algebra $A$ by a central element $b$, if $d_r(x)=b^r\cdot a$, then every element in $a\cdot E_{r+1}$ is a permanent cycle.
\end{lem}
\begin{proof}
The reason behind this lies in the construction of Bockstein spectral sequences. A $d_r$-differential $d_r(x)=ab^r$ means that in $H^\ast(A)$, $ab^r=0$ while $ab^{k}\neq 0$ for $k<r$. By associativity and commutativity of the multiplication in $H^\ast(A)$, we conclude that for all classes $c\in E_{r+1}$, $a\cdot c$ is annihilated by $b^r$. This means it can neither be the target of a $d_{s}$-differential (since it is killed by a smaller power of $b$) nor be the source of a $d_{s}$-differential (the target of which must be $b$-torsion free).
\end{proof}

\begin{Remark}
The associativity condition is essential. The homotopy-commutativity can be greatly weakened, but the statement of the lemma becomes much more complicated, since we must consider any product of elements on $E_{r+1}$, at least one of which is $a$. For our purposes, $A$ can be taken to be the cobar complex.
\end{Remark}

Computing $\Ext_{\A^\C}(\M_2^\C,\M_2^\C)$ is at least as difficult as computing the classical $\Ext$ groups. We therefore will look at two much simpler cases, illustrating the $\rho$-Bockstein spectral sequence by using it to compute  $\Ext_{E(n)}(\M_2,\M_2)$ for all $n$ and $\Ext_{\A(1)}(\M_2,\M_2)$.

The starting inputs are the cohomology of $E(n)^\C$ and $\A(1)^\C$.

\begin{prop}\label{prop:BocksteinSetup}\mbox{}
\begin{enumerate}
\item For $0\leq n\leq \infty$,
\[
\Ext_{E(n)^\C}(\M_2^\C,\M_2^\C)\cong\M_2^\C[v_0,\dots,v_n],
\]
where $\left| v_i\right|=(2^{i+1}-2,1,2^i-1)$.

\item As an algebra,
\[
\Ext_{\A(1)^\C}(\M_2^\C,\M_2^\C)\cong\M_2^\C[v_0, \eta, x, v_1^4]/(v_0\eta, \tau\eta^3, \eta x, x^2-v_0^2v_1^4),
\]
where $\left|\eta\right|=(1,1,1)$, $\left| x\right|=(4,3,2)$, and the classes $v_0$ and $v_1^4$ have the same tridegrees as above.
\end{enumerate}
\end{prop}
\begin{proof}
For $E(n)$, this is an immediate consequence of a change-of-rings theorem: the Hopf algebra $E(n)^\C$ is the classical Hopf algebra $E(n)$ base changed to $\M_2$. For $\A(1)^\C$\!, this computation can be done in many ways, and a resolution-based approach will appear in the thesis of Shkembi \cite{Sh09}.
\end{proof}

In the cohomology of $E(n)$, the classes $v_i$ are all classical. By Proposition~\ref{prop:PermCycles}, they are all permanent cycles. For $\A(1)$, the classes $v_0$, $\eta$, $x$, and $v_1^4$ are classical: they have cobar representatives with no reference to $\rho$ or $\tau$ (although we must be careful to distinguish between $\xi_1$ and $\tau_0^2$ when lifting the usual representatives). Proposition~\ref{prop:PermCycles} ensures that these are also all permanent cycles.

\section{$\Ext$ over $E(n)$}\label{sec:ExtE(n)}
\subsection{Bockstein Spectral Sequence}
Throughout this section, let $n$ be a positive integer or infinity. In this section, we will prove the following theorem.

\begin{thm}\label{thm:ExtE(n)}
As an algebra,
\[
\Ext_{E(n)}(\M_2,\M_2)=\F_2[\rho, \tau^{2^{n+1}}, v_i(j) | 0\leq i\leq n, 0\leq j]/ \rho^{2^{i+1}-1}v_i(j),
\]
and subject moreover to two additional relations:
\begin{enumerate}
\item If $i\geq k$, then $v_i(j)\cdot v_k(\ell)=v_i(j+2^{k-i}\ell)\cdot v_k(0)$,
\item and if $j\geq 2^{n-i},$ then $v_i(j)=\tau^{2^{n+1}} v_i(j-2^{n-i})$.
\end{enumerate}

Moreover, the class $v_i(j)$ is represented on $E_1$ by $\tau^{2^{i+1}j}v_i$.
\end{thm}

While the statement of the theorem is complicated, the method of computation is less so. We recall from Proposition~\ref{prop:BocksteinSetup} that the $E_1$-term of the Bockstein spectral sequence is
\[
E_1=\Ext_{E(n)^\C}(\M_2^\C,\M_2^\C)[\rho]=\F_2[\rho,\tau, v_0,\dots,v_n].
\]
Proposition~\ref{prop:PermCycles} guarantees that the classes $v_0$ through $v_n$ are all permanent cycles, and by construction, the class $\rho$ is as well. We therefore need only understand the differentials on $\tau$.

\begin{thm}\label{thm:E(n)Diffs}\mbox{}
\begin{enumerate}
\item For $i\leq n$, $E_{2^{i}}=\dots=E_{2^{i+1}-1}$ and
\[
E_{2^{i+1}-1}=\F_2[\rho, \tau^{2^{i}}\!, v_{j}(k), v_i,\dots v_n | 0\leq j\leq i-1, 0\leq k]/ \rho^{2^{j+1}-1}v_j(k),
\]
subject to two additional families of relations:
\[
v_{m}(r)\cdot v_{j}(\ell)=v_{m}(r+2^{j-m}\ell)\cdot v_{j}(0),
\]
if $m\leq j$, and if $r\geq 2^{i-m}$\!,
\[
v_{m}(r)=\tau^{2^{i+1}}v_{m}(r-2^{i-m}).
\]

\item The classes $v_{m}(r)$ are represented on $E_1$ by $\tau^{2^{m+1}r}v_m$ and are permanent cycles.

\item The $d_{2^{i+1}-1}$-differential is determined by
\[
d_{2^{i+1}-1}(\tau^{2^i})=\rho^{2^{i+1}-1}v_i.
\]
\end{enumerate}
\end{thm}

\begin{proof}
We prove this be induction on $i$. The base case of $i=0$ is the earlier analysis of the Bockstein $E_1$-term, and we already determined that $\tau$ is the only generator that is not a permanent cycle. We begin by showing the third part of the theorem. The first and second parts will follow easily from this and Lemma~\ref{lem:BocksteinCycles}.

The $d_k$-differential on $\tau^j$ is computed by reducing $\eta_L(\tau^j)-\eta_R(\tau^j)$ modulo $\rho^{k+1}$. The units are ring homomorphisms, so
\[
\eta_R\big(\tau^{2^i}\big)=\big(\eta_R(\tau)\big)^{2^i}=\tau^{2^i}+(\rho\tau_0)^{2^i}.
\]
Since in $E(n)$, $\tau_i^2=\rho\tau_{i+1}$ for $i<n$ and $\tau_n^2=0$, we conclude immediately that
\[
\eta_R\big(\tau^{2^i}\big)=\begin{cases}
\tau^{2^i}+\rho^{2^{i+1}-1}\tau_i & i\leq n, \\
\tau^{2^i} & \text{otherwise.}
\end{cases}
\]
The difference between this and the left unit gives the desired differential:
\[
\tau^{2^i}\xrightarrow{d_j} \begin{cases}
\rho^{2^{i+1}-1}v_i & i\leq n, j=2^{i+1}-1, \\
0 & \text{ otherwise.}
\end{cases}
\]

We can now complete our inductive argument. We assume that $E_{2^{i+1}-1}$ is of the stated form. By the above argument, there is a differential $d_{2^{i+1}-1}(\tau^{2^i})=\rho^{2^{i+1}-1}v_i$.

By the induction hypothesis, the classes annihilated by $\rho^{2^{i+1}-1}$ are in the ideal generated by the infinite families of permanent cycles $v_0(j),\dots,v_{i-1}(j)$. The collection of generators for this ideal is closed under $\tau^{2^i}$-multiplication, so the only new $d_{2^{i+1}-1}$-cycle is $\tau^{2^{i+1}}$. We remark here that this is the reason for our decision to include the redundant generators in our list of permanent cycles: the bookkeeping at this point is simplified.

Let $v_i(j)$ denote the class $\tau^{2^{i+1}j}v_i$. Since $v_i$ was truncated by $d_{2^{i+1}-1}$, Lemma~\ref{lem:BocksteinCycles} shows that for every $j\geq 0$, $v_i(j)$ is a permanent cycle. We therefore conclude that
\[
E_{2^{i+1}}=\F_2[\rho,\tau^{2^{i+1}}\!,v_j(k),v_{i+1},\dots,v_n|0\leq j\leq i, 0\leq k]/\rho^{2^{j+1}-1}v_j(k),
\]
subject additionally to the relations
\[
v_{m}(r)\cdot v_{j}(\ell)=v_{m}(r+2^{j-m}\ell)\cdot v_{j}(0),
\]
if $m\leq j$, and if $r\geq 2^{i-m}$\!,
\[
v_{m}(r)=\tau^{2^{i+1}}v_{m}(r-2^{i-m}).
\]

Since all of the generators with the exception of $\tau^{2^{i+1}}$ are permanent cycles, and since we have directly determined the differentials on $\tau^{2^{i+1}}$\!\!, we conclude that
\[
E_{2^{i+1}}=\dots E_{2^{i+2}-1},
\]
completing the proof of the first part.
\end{proof}

We remark that the ``obvious'' relations for the classes $v_i(j)$ are actually inherited from $E_1$: here we can drop the brackets and we just multiply as we normally do. We can therefore conclude these relations through $E_\infty$.

\begin{cor}
The $E_\infty$-page of the Bockstein spectral sequence for $\Ext_{E(n)}(\M_2,\M_2)$ is
\[
E_\infty=\F_2[\rho,\tau^{2^{n+1}}\!,v_i(j) | 0\leq i\leq n, 0\leq j]/\rho^{2^{i+1}-1}v_i(j),
\]
and subject to
\[
v_i(j)\cdot v_k(\ell)=v_i(j+2^{k-i}\ell)\cdot v_k(0)
\]
whenever $i\leq k$, and
\[
v_i(j)=\tau^{2^{n+1}}v_i(j-2^{n-i})
\]
whenever $j\geq 2^{n-i}$.

The tridegree of the element $v_i(j)$ is $(2^{i+1}-2, 1, 2^i-1-2^{i+1}j)$.
\end{cor}

The final issue to tackle is the resolution of hidden extensions. We have a family of multiplicative relations, but these are only {\em{a priori}} true modulo $\rho$. We must show that the last family of relations has no $\rho$-divisible correction factors. Better said, we must show that
\[
v_i(j)\cdot v_k(\ell)-v_i(j+2^{k-i}\ell)\cdot v_k(0)
\]
is zero. While the argument is not difficult (it amounts to elementary $2$-adic number theory), it is somewhat involved. Once this point is resolved, however, we have completed the proof of Theorem~\ref{thm:ExtE(n)}.

\subsection{Non-existence of Hidden Extensions}
\begin{thm}\label{thm:Extensions}
In $\Ext_{E(n)}(\M_2,\M_2)$,
\[
v_i(j)\cdot v_k(\ell)-v_i(j+2^{k-i}\ell)\cdot v_k(0)=0
\] whenever $k\geq i$.
\end{thm}
We argue this through a series of lemmata using tridegree considerations, explicitly determining that there is only a single non-zero class in the tridegree of $v_i(j)\cdot v_k(\ell)$. Since this has cohomological degree $2$ and since $\rho$ has cohomological degree $0$, we conclude that anything in this tridegree has the form $\rho^{m}v_r(s)\cdot v_q(t)$. In all that follows, we will assume we have constants $m$, $r$, $s$, $q$, $t$, $i$, $j$, $k$, and $\ell$ such that
\begin{equation}\label{eqn:MultExtension}
\left|\,\rho^{m}v_r(s)\cdot v_q(t)\right|=\left|\,v_i(j)\cdot v_k(\ell)\right|
\end{equation}
holds true. We may additionally assume without loss of generality that $r\geq q$ and $i\geq k$. We present without proof two small lemmas which allow us to conclude this result. The arguments are elementary, requiring only an analysis of the degrees and weights of the elements in question, and we leave them to the reader.

\begin{lem}\label{lem:k=i}
If Equation~\ref{eqn:MultExtension} holds, then $r=i$.
\end{lem}
%
%

\begin{lem}\label{lem:m=0}
If Equation~\ref{eqn:MultExtension} holds, then $q=k$, $m=0$, and
\[
2^{r+1}s+2^{q+1}t=2^{i+1}j+2^{k+1}\ell.
\]
\end{lem}

\begin{proof}[Proof of Theorem~\ref{thm:Extensions}]
Lemmata~\ref{lem:k=i} and \ref{lem:m=0} show that there is a unique class in the tridegree under consideration. We therefore conclude that the multiplication in the associated graded is the correct one.
\end{proof}

We close the section with a picture of $\Ext_{E(1)}(\M_2,\M_2)$. This has the advantage of simultaneously being non-trivial (unlike $\Ext_{E(0)}(\M_2,\M_2)$) and small enough to easily draw (unlike $\Ext_{E(2)}(\M_2,\M_2)$ and beyond). This computation is also used in our subsequent computations. In the picture, the horizontal axis shows $t-s$ while the vertical axis gives $s$. A black dot represents $\F_2[\tau^4]$.

\begin{figure}[ht]
\includegraphics{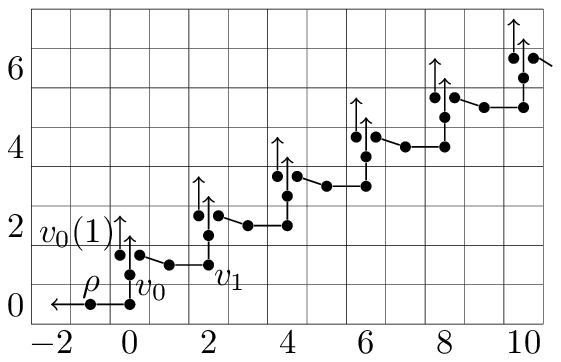}
\caption{$\Ext_{E(1)}(\M_2,\M_2)$}
\label{fig:ExtE(1)}
\end{figure}

\section{$\Ext$ over $\A(1)$}\label{sec:ExtA(1)}

Pictures are indispensable in the computation of $\Ext_{\A(1)}$. Since our spectral sequences are quadruply graded, we must choose two gradings for our axes. We will draw our figures like Adams charts: the horizontal axis is the topological degree $t-s$ and the vertical axis is the cohomological degree $s$. We suppress explicit mention of both the motivic weight and the Bockstein degree.

By construction, the target of a $d_r$-differential is divisibly by $\rho^r$. We will use a solid, horizontal line to denote multiplications by $\rho$, so the only possible targets for a $d_r$-differentials are classes which are followed on the right by exactly $r$ solid lines (thus the Bockstein degree is implicitly reflected in the picture). Moreover, the Bockstein differentials preserve the internal degree and the weight while increasing the filtration by $1$. This means that in our pictures, the Bockstein differentials appear as Adams $d_1$-differentials. These two observations about the structure and targets of differentials greatly simplifies the story.

Given the plethora of elements in all of the figures, we also adopt a convention: any class which is $\rho$-torsion is not drawn on following pages. This corresponds to running the $\rho$-Bockstein spectral sequence with $\rho$ inverted, uncluttering the picture. The algebraic machinery of the spectral sequence keeps track of the $\rho$-torsion elements for us.

In all of the pictures that follow, and open circle represents a polynomial algebra on $\rho$ (times whatever is linked to the open circle by horizontal lines). Black dots are polynomial algebras in $\tau$, and stars are classes which are simple $\tau$-torsion. Circled classes are algebra generators.

\subsection{The $E_1$-page}
We present a picture of the $E_1$-page in Figure~\ref{fig:ExtE1R}.

\begin{figure}[h]
\includegraphics{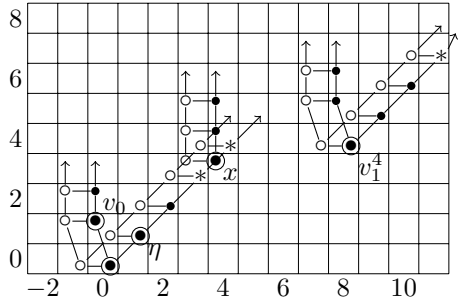}
\caption{The $\rho$-Bockstein $E_1$-Page}
\label{fig:ExtE1R}
\end{figure}

\begin{prop}
The $d_1$-differential is determined by
\[
d_1(\tau)=v_0\rho.
\]
\end{prop}
\begin{proof}
The $d_1$-differential is given by the difference of the left and right units on $\tau$ and reducing modulo $\rho^2$. For degree reasons, and from our earlier observation about permanent cycles, this is the only generator supporting a $d_1$.
\end{proof}

This produces two $d_1$-cycles: $\tau^2$ and $\eta_0=\eta \tau$.

\begin{prop}
The class $\eta_0$ is a permanent cycle.
\end{prop}
\begin{proof}
The class $\eta_0$ is represented by $\tau_0^2$ in the cobar complex. Since $\tau_0$ is primitive, so is this class, and we conclude that it is a permanent cycle.
\end{proof}

We pause here to identify some extensions that increase the Bockstein filtration by $1$. These arise from Massey products built out of $\tau$. From $E_2$ on, the class $\tau$ no longer makes sense, so multiplication by it is ill-defined. However, many of the $\tau$ multiples of classes persist, and we can represent these as a Massey product:
\[
\tau\cdot(-)=\toda{\rho, v_0, -}.
\]
In particular, we learn that the class $\eta_0$ is the bracket $\toda{\rho, v_0, \eta}$ (this is also immediate from the form of $\eta_0$).
\begin{prop}\label{prop:HiddenExtensions}
We have hidden multiplicative extensions
\[
v_0\eta_0=\rho\eta\eta_0
\]
and
\[
\eta\eta_0^2=\rho x.
\]
\end{prop}
\begin{proof}
These follow from standard shuffling results. The $v_0$-multiplication is given by
\[
\eta_0\cdot v_0=\toda{\rho,v_0,\eta}\cdot v_0=\rho\cdot\toda{v_0,\eta,v_0}=\rho \eta_0\eta,
\]
where the last bracket, the motivic analogue of $\toda{2,\eta,2}=\eta^2$, follows from an elementary computation in the cobar complex.

In a similar vein, the hidden $\eta$-multiplication follows from:
\[
\eta_0^2\eta=\eta_0\cdot(\eta_0\eta)=\toda{\rho, v_0, \eta}\cdot\eta_0\eta=\rho\toda{v_0,\eta, \eta_0\eta}=\rho x.\qedhere
\]
\end{proof}
These relations will be essential for later differentials and cycles. In particular, we conclude immediately that modulo $\rho^2$,
\[
\tau^4\eta^3=\rho\tau^2x.
\]

\subsection{The $E_2$-page}
The $E_2$-term is given by
\begin{multline*}
\F_2[\tau^2, v_0, \eta, \eta_0, x, v_1^4, \rho]/(v_0\rho, v_0\eta, \tau^2\eta^3, \eta x, x^2-v_0^2v_1^4, \eta_0 \eta^2, \eta_0^2-\tau^2\eta^2, \eta_0x, v_0\eta_0).
\end{multline*}
We present a picture in Figure~\ref{fig:ExtE2R}. Here a dot indicates a polynomial algebra on $\tau^2$.

\begin{figure}[h]
\includegraphics{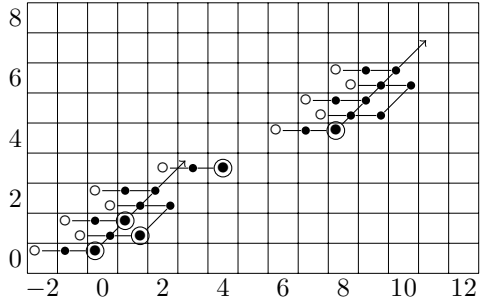}
\caption{The $\rho$-Bockstein $E_2$-Page}
\label{fig:ExtE2R}
\end{figure}

\begin{prop}
The $d_2$-differential is determined by
\[
d_2(\tau^2)=\rho^2\eta_0.
\]
\end{prop}
\begin{proof}
With the exception of $\tau^2$, all of the algebra generators are permanent cycles. The differential on $\tau^2$ is immediate from the cobar complex: the left and right units are algebra maps.
\end{proof}

We remark that we can find this differential by applying the power operation $Sq^0$:
\[
d_2(Sq^0(\tau))=Sq^0(d_1(\tau))=Sq^0(v_0)Sq^0(\rho)=\tau\eta\rho^2=\eta_0\rho^2.
\]
The two approaches are essential the same; the latter has the possibility of hiding some of the higher filtration classes that arise when we complete a given $d_i$-cycle to an honest cohomology class.

Since $\eta_0$ is $v_0$, $\eta^2$, $\eta\eta_0$ and $x$ torsion (at least modulo $\rho$), we get new cycles: $\tau^4$, $a=v_0\tau^2$, $\eta^2 \tau^2$ (which is $\eta_0^2$), $c=\eta\eta_0\tau^2$, and $b=x\tau^2$. We remark that Lemma~\ref{lem:BocksteinCycles} does not apply to the class $c$, since $c$ is not $\eta_0$-divisible on $E_3$.

\begin{prop}
There is a hidden $\eta$-multiplication:
\[
a\eta=\rho\eta_0^2.
\]
\end{prop}
\begin{proof}
From $E_3$ on, $\tau^2$ multiplication is no longer well defined. Instead, we have a bracket formulation:
\[
\tau^2\cdot(-)=\toda{\eta_0,\rho^2,-}\text{ or }\toda{\rho^2, \eta_0,-}.
\]
Since $v_0$ is $\rho$-torsion, we use the former, rearranging slightly to see that
\[
a=\toda{\eta_0\rho, \rho, v_0}.
\]
If we multiply by $\eta$, then we get
\[
a\eta=\toda{\eta_0\rho,\rho,v_0}\eta=\eta_0\rho\toda{\rho,v_0,\eta}=\rho\eta_0^2,
\]
where the last equality follows from the definition of $\eta_0$.
\end{proof}
We remark that as with most statements, it is also easy to give a justification of this using explicit cycles in the cobar complex. The class $\tau^2v_0$ is represented by
\[
\tau^2\tau_0+\tau\rho\tau_0^2+\rho^2\tau_0^3,
\]
and multiplication by $\xi_1$ can be easily rearranged to give $\rho\tau_0^2|\tau_0^2$.

Just as before, we can work out additional hidden extensions:
\[
v_0(\tau^2\eta^2)=\toda{\rho^2,\eta_0, \eta^2}v_0=\rho^2\toda{\eta_0, \eta^2, v_0}=\rho^2 x.
\]
This is consistent with the previous result:
\[
v_0(\eta_0^2)=(\rho\eta\eta_0)\cdot\eta_0=\rho\cdot(\rho x).
\]

These relations actually immediately tell us that we are not finished. The classes $v_0$ and $\eta_0$ are $\rho$-torsion while $x$ is not.

\subsection{The $E_3$-page}
We have the following $E_3$-page:
\begin{multline*}
\F_2[\tau^4, \rho, v_0, a, \eta, \eta_0, c, x, b, v_1^4]/(
v_0\rho, v_0\eta, v_0\eta_0, v_0c,
a\rho, a\eta, a\eta_0, ac,
\eta x, \eta_0x, \eta b, \eta c, \eta \eta_0^2,\\ \tau^4\eta^3,
\eta_0b, \eta_0 c, \eta_0\eta^2, \eta_0^3,
cx, cb, x^2-v_0^2v_1^4, xb-av_0v_1^4, b^2-a^2v_1^4,
v_0^2\tau^4-a^2, \tau^4v_0x-ab).
\end{multline*}

This is presented in Figure~\ref{fig:ExtE3R}. Here black dots are polynomial algebras on $\tau^4$.
\begin{figure}[h]
\includegraphics{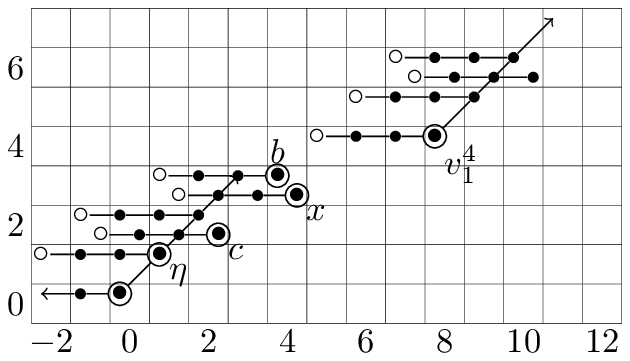}
\caption{The $\rho$-Bockstein $E_3$-Page}
\label{fig:ExtE3R}
\end{figure}

\begin{prop}
There is a $d_3$-differential
\[
d_3(c)=\rho^3x.
\]
\end{prop}
\begin{proof}
The class $c$ was a $d_2$-cycle because $\eta_0^2\eta$ was zero modulo $\rho$. However, Proposition~\ref{prop:HiddenExtensions} shows that $\eta_0^2\eta=\rho x$. Thus the differential on $c$ actually takes the form
\[
c\mapsto \rho^2\eta_0^2\eta=\rho^3 x.\qedhere
\]
\end{proof}

For degree reasons, there are no more $d_3$-differentials. Additionally, since almost all classes annihilated $c$ (indeed, these are all of the classes annihilated by $\rho^3 x$), we produce no new cycles. This in turn means we do not have new multiplicative relations.

\subsection{The $E_\infty$-page and $\Ext$}
This allows us to produce the $E_4$-page:
\begin{multline*}
\F_2[\tau^4, \rho, v_0, a, \eta, \eta_0, x, b, v_1^4]/(
v_0\rho, v_0\eta, v_0\eta_0,
a\rho, a\eta, a\eta_0,
\eta x, \eta b, \eta \eta_0^2, \tau^4\eta^3,
\eta_0x, \eta_0b,\\ \eta_0\eta^2, \eta_0^3, \rho^3x,
x^2-v_0^2v_1^4, xb-av_0v_1^4, b^2-a^2v_1^4,
v_0^2\tau^4-a^2, \tau^4v_0x-ab).
\end{multline*}

The class $b$ is a permanent cycle for a very simple reason: $\rho b$ is a permanent cycle, so $b$ must also be a permanent cycle (this is an artifact of the fact that higher Bockstein differentials never {\em{further}} truncate classes).

This class is also present in a hidden extension not readily predicted from shuffling considerations.

\begin{prop}
There is a hidden $\eta$-multiplication of the form
\[
\eta\cdot b=\rho^3v_1^4.
\]
\end{prop}
\begin{proof}
We will prove this by comparing with $\Ext_{E(1)}$. Let $C(\eta)$ denote the subcomodule algebra of $\A(1)$ generated by $\xi_1$. A change-of-rings argument shows that
\[
\Ext_{\A(1)}\big(\M_2,C(\eta)\big)=\Ext_{\A(1)}\big(\M_2,\A(1)\Box_{E(1)}\F_2\big)=\Ext_{E(1)}(\M_2,\M_2),
\]
so we can interpret the effect in $\Ext$ of the canonical quotient map $\A(1)\to E(1)$ as the reduction modulo $\eta$.

The induced map
\[
\Ext_{\A(1)}(\M_2,\M_2)\to \Ext_{E(1)}(\M_2,\M_2)
\]
takes $v_1^4$ to $v_1^4$. By Theorem~\ref{thm:ExtE(n)}, $\rho^3v_1^4=0$ in $\Ext_{E(1)}(\M_2,\M_2)$, so we conclude that the class $\rho^3v_1^4$ must be divisible by $\eta$ in $\Ext_{\A(1)}(\M_2,\M_2)$. The only possible relation is the given one, for tridegree reasons.
\end{proof}

We remark that from this, we get an amusing formula:
\[
\tau^4\eta^4=\rho^4v_1^4.
\]
This is a strange cobar analogue of $\tau_0^4=0$.

\begin{thm}
As a ring
\begin{multline*}
\Ext_{\A(1)^\R}(\M_2,\M_2)=
\F_2[\tau^4, \rho][v_0, a, \eta, \eta_0, x, b, v_1^4]/\\(
v_0\rho, v_0\eta,
a\rho, a\eta_0,
\eta x, \tau^4\eta^3,
\eta_0x, \eta_0b, \eta_0\eta^2, \eta_0^3, \rho^3x,
x^2-v_0^2v_1^4, ab-\tau^4v_0x, \\ xb-av_0v_1^4, b^2-a^2v_1^4,
a^2-v_0^2\tau^4, \eta b-\rho^3v_1^4, v_0\eta_0-\rho\eta\eta_0, \eta \eta_0^2-\rho x, a\eta-\rho\eta_0^2).
\end{multline*}
\end{thm}

\section{Applications}
Hopkins has remarked that the motivic Adams spectral sequence converges to the completion with respect to $2$ and $\eta$. This has recently be verified by Hu-Kriz-Ormsby \cite{HuKrOr}, who show that with appropriate connectivity hypotheses, the $\eta$-completion is unnecessary. We will apply this and the previous computations to determine the $2$-completed motivic homotopy of $BPGL$ over $\R$ and that of a conjectural connective spectrum $ko^\R$. We begin with a general remark about comparison with the complex and classical cases. Working over $\Spec(\C)$, Shkembi has produced spectra $ko^\C$ and $ku^\C$ which have the desired homotopy \cite{Sh09}.

There is a complexification map from the real Adams spectral sequence to the complex Adams spectral sequence, and the map on the $E_2$-terms is reduction modulo $\rho$. For $BPGL$, the complex Adams spectral sequence collapses since, just as in the classical case, the $E_2$-term is concentrated in even topological degrees. For the conjectural spectrum $ko^\R$, comparison with the classical case also shows that the complex Adams spectral sequence collapses.

In both cases, the collapse of the complex Adams spectral sequence puts large restrictions on the possible Adams differentials and multiplicative extensions. By naturality of the Adams spectral sequence, we conclude that the targets of any Adams differential for the real spectral sequence must be divisible by $\rho$, as must the correction terms for any multiplicative extensions. In the $ko^\R$ case, this greatly simplifies computations.

\subsection{The Adams Spectral Sequence for $BPGL$}
Let $MGL$ be the Thom spectrum of the universal bundle over $BGL$, and let $BPGL$ be the standard summand of the $2$-localization. Hu-Kriz conjectured the cohomology of $BPGL$ \cite{HuKr01b}.
\begin{conj}
As an $\A$-comodule algebra,
\[
H_{\ast,\ast}(BPGL;\F_2)=\A\Box_{E(\infty)}\M_2.
\]
\end{conj}

A standard change-of-rings argument then shows us that the $E_2$-term of the Adams spectral sequence is a ring we computed in \S\ref{sec:ExtE(n)}.

\begin{cor}
As an algebra,
\[
E_2=\Ext_{\A}\big(\M_2,H_{\ast,\ast}(BPGL)\big)=\Ext_{E(\infty)}\big(\M_2,\M_2).
\]
\end{cor}

A combinatorial analysis of the tridegrees shows the following theorem.

\begin{thm}
The Adams spectral sequence for $BPGL$ collapses at $E_2$.
\end{thm}

Just as classically, the element $v_0(0)$ detects multiplication by $2$. Over $\R$, this is not a trivial statement. In $\pi_{0,0}S^{0,0}$, there are homotopy classes represented by $v_0(0)$ and $\rho\eta_0$, and the element $2$ is represented by the sum of these classes. Since $\rho\eta_0=0$ here, $v_0(0)$ multiplication detects multiplication by $2$. If we can resolve the possibility of extensions, then we have computed the homotopy of $BPGL$.

\begin{thm}
There are no hidden multiplicative extensions.
\end{thm}

We argue this using the realification functor $t$ from the motivic stable homotopy category over $\R$, $\mathcal S_\R$, to the $\Z/2$-equivariant stable homotopy category, $\mathcal S^{\Z/2}$, described by Hu-Kriz \cite{HuKr01, HuKr01b}. This functor induces a map of bigraded homotopy rings
\[
\pi_{\ast,\ast}(X)\to\pi_{\ast,\ast}(t(X)).
\]
In the target, the bigraded homotopy groups are actually graded by the real representation ring of $\Z/2$. Hu and Kriz showed that $t(BPGL)=BP\R$, the summand of $2$-localization of the cobordism spectrum representing Real manifolds. They also computed the $RO(\Z/2)$-graded homotopy groups of $BP\R$. This ring is generated by classes $v_i(j)$ for $0\leq i$ and for $j\in\Z$ and a class $a$, subject to the relations identical to those in \S\ref{sec:ExtE(n)}:
\[
v_i(j)\cdot v_k(\ell)=v_i(j+2^{k-i}\ell)\cdot v_k(0)
\]
when $i\leq k$, and $a^{2^{i+1}-1}v_i(j)=0$.

The realification functor takes $\rho$ to $a$ and takes $v_i(j)$ to $v_i(j)$. We therefore conclude that this embeds $\pi_{\ast,\ast}BPGL$ into $\pi_{\ast,\ast}BP\R$. In particular, we also conclude that there are no multiplicative extensions and we have computed the homotopy of $BPGL$.

\subsection{The Adams Spectral Sequence for $ko^\R$}
In this section, we assume the existence of a motivic spectrum $ko^\R$ which has
\[
H_{\ast,\ast}(ko^\R)=\A\Box_{\A(1)}\M_2.
\]
The aforementioned change-of-rings argument tells us that we computed the Adams $E_2$-term for $ko^\R$ in \S\ref{sec:ExtA(1)}.

\begin{cor}
As an algebra,
\[
E_2=\Ext_{\A}\big(\M_2,H_{\ast,\ast}(ko^{\R})\big)=\Ext_{\A(1)}(\M_2,\M_2).
\]
\end{cor}

To determine if there are Adams differentials, we again argue by sparceness.

\begin{thm}
The Adams spectral sequence for $\pi_{\ast,\ast}ko^\mathbb R$ collapses at $E_2$.
\end{thm}

This is turn should have applications to both Hermitian $K$-theory and the motivic image of $J$. If we invert the class $v_1^4$, the class corresponding to the Bott class, then we get a non-connective, periodic ring which should be the motivic homotopy ring of Hermitian $K$-theory. More geometric computations are needed here. This is analogous to what would happen with a connective complex $K$-theory spectrum $ku$.

%

Classically, the image of $J$ can be detected using real $K$-theory. We learn from this computation that if there is a spectrum $ko$, then the image of $J$ elements in $\pi_{\ast,\ast}S^0$ are all $\rho$-torsion free. Let $P(-)$ denote Adams' periodicity operator. Then $P^i(\eta^k)\in\pi_{8i+k,4i+k}S^0$ will map to the element $v_1^i\eta^k\in\pi_{8i+k,4i+k}ko$, and this element is $\rho$-torsion free. For the class $\eta$, this is not surprising. This map can be explicitly constructed, and since $\rho$ survives the Adams spectral sequence, we learn that $\rho^k\eta$ survives. For higher classes it is less trivial. In particular, we see that the classes $\rho^nP^k(\eta)$ survive for all $n$ and $k$.

\bibliography{bib}

\end{document}